\documentclass[12pt]{article}

\usepackage{amsmath,amssymb,amscd,epic}
\usepackage{bbm,amsthm}
\usepackage{mathtools}
\usepackage{leftidx}
\usepackage{hyperref}
\usepackage{mathrsfs}
\usepackage{here}
\usepackage{lscape}
\usepackage{xcolor}
\usepackage[top=30truemm,bottom=30truemm,left=25truemm,right=25truemm]{geometry}

\usepackage{caption}
\usepackage{subcaption}

\usepackage{ytableau}
\ytableausetup{boxsize=1.2em, centertableaux}

\usepackage[all,pdf]{xy}

\theoremstyle{plain}
\newtheorem{thm}{Theorem}[section]
\newtheorem{lem}[thm]{Lemma}

\newtheorem{cor}[thm]{Corollary}

\theoremstyle{definition}

\newtheorem{rem}[thm]{Remark}

\newcommand{\perm}{{\rm perm}}

\newcommand{\Per}{{\rm P}}

\newcommand{\Nu}{{\rm N}}
\newcommand{\ord}{{\rm ord}}

\title{Group Permanents of Abelian $p$-Groups \\ and Young Diagrams}
\author{Naoya Yamaguchi and Yuka Yamaguchi}
\date{}

\begin{document}

\maketitle

\begin{abstract}
We study the number $\Nu(\Per(G_{\lambda}))$ of distinct monomials with nonzero coefficients in the group permanent of an abelian $p$-group $G_\lambda$ associated with a partition $\lambda$ of a positive integer $N$. 
First, we derive an explicit formula for $\Nu(\Per(G_{\lambda}))$ in terms of the partial column sums of the Young diagram of $\lambda$. 
Next, we show that the relative order of the values $\Nu(\Per(G_{\lambda}))$ is determined by a lexicographic comparison of the conjugate Young diagrams. 
Finally, we investigate congruence properties of $\Nu(\Per(G_{\lambda}))$ for abelian $p$-groups and establish a criterion involving Wolstenholme primes. 
\end{abstract}


\section{Introduction}
\label{sec:intro}

The circulant determinant is a classical object; see, for example, Ore~\cite{MR42365}. 
For a finite group $G$, 
let 
\[
M(G) := (x_{g h^{-1}})_{g, h \in G}
\]
denote the group matrix associated with $G$. 
The determinant 
\[
\Theta(G) := \det{(M(G))}
\]
is called the group determinant of $G$. 
When $G$ is the cyclic group ${\rm C}_n$, the group determinant $\Theta(G)$ coincides with the circulant determinant.
The group determinant was introduced by Dedekind, and its irreducible factorization over the complex numbers was later obtained by Frobenius~\cite{Frobenius1968gruppen}. 
For historical background, we refer the reader to \cite{MR1659232, MR1554141, van2013history}. 

In this paper, we consider the permanent of the group matrix $M(G)$. 
The permanent 
\[
\Per(G) := \perm{(M(G))} 
\]
is called the group permanent of $G$. 
For a polynomial $f$, 
let $\Nu(f)$ denote the number of distinct monomials with nonzero coefficients in $f$. 
We investigate $\Nu(\Per(G))$. 
For an abelian group $G$ of order $n$, 
Hall~\cite{MR50579} proved that 
\begin{align}\label{eq:Hall}
\Nu(\Per(G)) = |M(G,n)|, 
\end{align}
where 
\[
M(G,m) := \{ (h_{1}, h_{2}, \ldots, h_{m}) \in G^{m} \mid h_{1} h_{2} \cdots h_{m} = 1 \}. 
\]
This equality is also mentioned in \cite[p.~121]{MR2747803}. 
Thus, in the abelian case, determining $\Nu(\Per(G))$ reduces to determining the cardinality of $M(G,n)$. 

When $G$ is cyclic, Brualdi--Newman~\cite{MR266850} gave the explicit formula
\[
\Nu(\Per({\rm C}_{n})) 
= \frac{1}{n} \sum_{d \mid n} \varphi\left( \frac{n}{d} \right) \binom{2d-1}{d}, 
\]
where $\varphi$ denotes Euler's totient function. 
Moreover, for an abelian group $G$ of order $n$, 
it was shown in \cite{MR4261108, MR3448603} that
\begin{equation}\label{eq:2}
|M(G,m)|
=\frac{1}{m+n} \sum_{d \mid \gcd(m,n)} \varphi_G(d)\binom{\frac{m+n}{d}}{\frac{m}{d}},
\end{equation}
where
\[
\varphi_G(d) := \left| \{ g\in G \mid \ord(g) = d \} \right| 
\]
denotes the number of elements of order $d$ in $G$. 
In particular, by taking $m=n$ and using equations~\eqref{eq:Hall} and~\eqref{eq:2}, 
we obtain 
\begin{equation}\label{eq:3}
\Nu(\Per(G)) = \frac{1}{2n} \sum_{d \mid n} \varphi_G(d) \binom{\frac{2n}{d}}{\frac{n}{d}}.
\end{equation}
Explicit formulas for the number of terms in $\Per({\rm C}_n)^k$ are also known; see \cite[Theorem~1.6]{yamaguchi2025principalspecializationmonomialsymmetric}:
\[
\Nu(\Per({\rm C}_{n})^{k})
=\frac{1}{n}\sum_{d\mid n}\varphi\left(\frac{n}{d}\right) \binom{dk+d-1}{d-1}.
\]

For cyclic $p$-groups, several results concerning the numbers of terms in the group determinant and the group permanent are known.
Thomas~\cite{MR2104820} proved that
\[
\Nu(\Theta({\rm C}_{p^N}))
=
\Nu(\Per({\rm C}_{p^N})) 
\]
for $N\ge 1$. 
Thus, congruence properties of $\Nu(\Theta({\rm C}_{p^N}))$ immediately yield corresponding results for $\Nu(\Per({\rm C}_{p^N}))$. 
For an odd prime $p$, it was shown in \cite{MR4593070} that
\[
\Nu(\Theta({\rm C}_p))\equiv 1 \pmod{p}. 
\]
Furthermore, \cite{MR4816571} proved that
\[
\Nu(\Theta({\rm C}_p))\equiv 1 \pmod{p^2}
\]
for every prime $p\ge 5$, and that
\[
\Nu(\Theta({\rm C}_p))\equiv 1 \pmod{p^3} 
\iff
p \text{ is a Wolstenholme prime}. 
\]
Recall that a prime $p$ is called a Wolstenholme prime if 
\[
\binom{2p-1}{p-1} \equiv 1 \pmod{p^{4}}.
\]
There are several equivalent formulations of this condition, usually involving harmonic sums~\cite[p.~385]{MR1339137}. 
More recently, 
\cite[Corollary~1.7]{yamaguchi2025principalspecializationmonomialsymmetric} established congruence properties for the numbers of terms in powers of the group permanent. Specifically, for a prime $p\ge 5$ and integers $N,k\ge 1$,
\[
\Nu(\Per({\rm C}_{p^N})^k) \equiv 1 \pmod{p^{2}}.  
\]
It was also shown that, 
for every $N\geq 1$, Wolstenholme primes give a sufficient condition for the congruence
\[
\Nu(\Per({\rm C}_{p^N})) \equiv 1 \pmod{p^{3}}
\]
to hold. 
We strengthen this result by proving the converse.
More precisely, for every $N\geq 1$, we show that
\[
\Nu(\Per({\rm C}_{p^N})) \equiv 1 \pmod{p^{3}}
\]
if and only if $p$ is a Wolstenholme prime.

Recently, Li--Zhang~\cite{MR4826770} proved that, 
for finite abelian groups $G$ and $H$,
\[
\Nu(\Per(G)) = \Nu(\Per(H)) \iff G \cong H,
\]
showing that the number of terms in the group permanent determines the isomorphism class of a finite abelian group. 
They also proved that if $G$ is a noncyclic abelian group of order $n$, then
\[
\Nu(\Per(G)) > \Nu(\Per({\rm C}_n)),
\]
thereby answering a question of Panyushev~\cite{MR2747803}. 
Thus, the number $\Nu(\Per(G))$ reflects the structure of the underlying finite abelian group. 
A natural question is how the structure of an abelian $p$-group is reflected in the relative order of the values $\Nu(\Per(G))$. 
This relationship, however, has not yet been fully understood. 

We now recall the basic partition-theoretic notation used throughout the paper. 
For a positive integer $N$, 
we write $\lambda \vdash N$ if $\lambda$ is a partition of $N$, 
that is, 
\[
\lambda = (\lambda_1, \lambda_2, \ldots, \lambda_r),
\qquad
\lambda_1 \geq \lambda_2 \geq \cdots \geq \lambda_r \geq 1,
\qquad
N = \sum_{i = 1}^r \lambda_i.
\]
For a partition $\lambda$, 
let $\lambda' = (\lambda'_1, \lambda'_2, \ldots )$ denote its conjugate partition, 
obtained by transposing the corresponding Young diagram, 
where $\lambda'_{t} := 0$ for $t > \lambda_{1}$. 
For partitions $\lambda$ and $\mu$ of the same integer, we use the lexicographic order: we write $\lambda<\mu$ if there exists $m$ such that $\lambda_i=\mu_i$ for all $i<m$ and $\lambda_m<\mu_m$.

To a partition $\lambda=(k_{1},k_{2},\ldots,k_{r})$, 
we associate the abelian $p$-group
\[
G_{\lambda}
:=
{\rm C}_{p^{k_{1}}}\times {\rm C}_{p^{k_{2}}}\times \cdots \times {\rm C}_{p^{k_{r}}}.
\]
By the structure theorem for finite abelian groups, every abelian $p$-group is represented uniquely in this form. 
For a partition $\lambda = (k_{1}, k_{2}, \ldots, k_{r}) \vdash N$, we define
\[
s_t(\lambda):=\sum_{i=1}^{r}\min(k_i,t), \qquad t = 1, 2, \ldots, N, 
\]
and $s_0(\lambda):= 0$. 
This is the number of boxes in the first $t$ columns of the Young diagram of $\lambda$. 
Equivalently,
\[
s_t(\lambda)=\lambda'_1+\lambda'_2+\cdots+\lambda'_t.
\]

Our main results are as follows.

\begin{thm}\label{thm:1}
Let $\lambda$ be a partition of a positive integer $N$, and let $G_{\lambda}$ be the corresponding abelian $p$-group. Then
\[
\Nu(\Per(G_{\lambda}))
=
1+\frac{1}{2p^{N}}\sum_{t=0}^{N-1}p^{s_{t}(\lambda)}
(\beta_{t}-\beta_{t+1}),
\qquad
\beta_{t}:=\binom{2p^{N-t}}{p^{N-t}}.
\]
\end{thm}

Theorem~\ref{thm:1} gives an explicit description of $\Nu(\Per(G_{\lambda}))$ in terms of the partial column sums of the Young diagram of $\lambda$. 
It forms the basis for the order comparison and congruence results obtained later. 

\begin{thm}\label{thm:4}
Let $\lambda$ and $\mu$ be partitions of a positive integer $N$. 
Assume that there exists $m\in\{1, 2, \ldots, N-1\}$ such that
\[
s_t(\lambda)=s_t(\mu)\quad (t=1,2,\ldots,m-1),
\qquad
s_m(\lambda)<s_m(\mu).
\]
Then
\[
\Nu(\Per(G_{\lambda})) < \Nu(\Per(G_{\mu})).
\]
\end{thm}

Since $s_t(\lambda)=\lambda'_1+\lambda'_2+\cdots+\lambda'_t$, 
the assumptions in Theorem~\ref{thm:4} mean precisely that
\[
\lambda' < \mu'
\]
in the lexicographic order on the conjugate partitions. 
Thus, Theorem~\ref{thm:4} shows that the relative order of the values $\Nu(\Per(G_{\lambda}))$ is governed by the lexicographic order on the conjugate Young diagrams.

In contrast to the known congruences for the number of terms in the group determinants of cyclic groups, the following two theorems provide congruences for the number of terms in the group permanents of noncyclic abelian $p$-groups.

\begin{thm}\label{thm:5}
Let $p \geq 5$ be a prime. 
Assume that $G_{\lambda}$ is a noncyclic abelian $p$-group. 
Then
\[
\Nu(\Per(G_{\lambda})) \equiv 1 \pmod{p^{3}}. 
\]
\end{thm}

\begin{rem}
When $p = 3$ and $\lambda = (1,1)$, by Theorem~\ref{thm:1}, we have
\[
\Nu(\Per(G_{\lambda})) = 2710 \equiv 10 \not\equiv 1 \pmod{3^{3}}.
\]
Thus the hypothesis $p\geq 5$ cannot be omitted.
\end{rem}

\begin{thm}\label{thm:6}
Let $p \geq 5$ be a prime. 
Assume that $G_{\lambda}$ is a noncyclic abelian $p$-group. 
Then
\[
\Nu(\Per(G_{\lambda})) \equiv \binom{2p-1}{p-1} \pmod{p^{4}}.
\]
\end{thm}

As an immediate consequence, 
we obtain the following criterion in terms of Wolstenholme primes.

\begin{cor}\label{cor:}
Let $p \geq 5$ be a prime. 
Assume that $G_{\lambda}$ is a noncyclic abelian $p$-group. 
Then
\[
\Nu(\Per(G_{\lambda})) \equiv 1 \pmod{p^{4}} 
\iff
p \text{ is a Wolstenholme prime}.
\]
\end{cor}

\begin{thm}\label{thm:7}
Let $p \geq 5$ be a prime and let $N\geq 1$. Then
\[
\Nu(\Per({\rm C}_{p^{N}})) - 1 \equiv \frac{1}{p} \left\{ \binom{2p-1}{p-1} - 1 \right\}  \pmod{p^{4}}. 
\]
Here the right-hand side is an integer by Wolstenholme's theorem. 
\end{thm}

Theorem~\ref{thm:7} immediately gives the following corollary. 

\begin{cor}\label{cor:8}
Let $p \geq 5$ be a prime and let $N\geq 1$. Then 
\[
\Nu(\Per({\rm C}_{p^{N}})) - 1 \equiv \frac{1}{p} \left\{ \binom{2p-1}{p-1} - 1 \right\} \pmod{p^{3}}. 
\]
Consequently, 
\[
\Nu(\Per({\rm C}_{p^{N}})) \equiv 1 \pmod{p^{3}}
\iff
p \text{ is a Wolstenholme prime}.
\]
\end{cor}

The paper is organized as follows. 
In Section~\ref{sec:main}, 
we derive the explicit formula in Theorem~\ref{thm:1}. 
In Section~\ref{sec:order}, we establish the order comparison theorem, namely Theorem~\ref{thm:4}. 
In Section~\ref{sec:wolstenholme}, 
we investigate congruence properties of $\Nu(\Per(G_{\lambda}))$ and prove Theorems~\ref{thm:5}, \ref{thm:6}, and~\ref{thm:7}.

\section{An explicit formula for $\Nu(\Per(G_{\lambda}))$}
\label{sec:main}

In this section, we prove Theorem~\ref{thm:1}. 
The key point is to express the number of elements of each possible order in $G_{\lambda}$ in terms of the partial column sums $s_t(\lambda)$ of the Young diagram of $\lambda$. 
We first record two elementary lemmas on abelian $p$-groups.

For any abelian $p$-group $G$, we define
\[
\alpha_t(G) := \bigl| \{g\in G \mid g^{p^t}=e\} \bigr|, 
\qquad t\ge 0, 
\]
where $e$ denotes the identity element of $G$. 
We also set $\alpha_{-1}(G):=0$. 
Let $G_{\lambda}$ be the abelian $p$-group of order $p^{N}$ corresponding to a partition $\lambda$ of $N$. 

\begin{lem}\label{lem:5}
For every $t=0,1,\ldots,N$, we have
\[
\alpha_t(G_\lambda)=p^{s_t(\lambda)}.
\]
\end{lem}

\begin{proof}
Let $\lambda=(k_{1},k_{2},\ldots,k_{r})$. 
By the direct product structure,
\[
\alpha_t(G_\lambda)=\prod_{i=1}^r \alpha_t({\rm C}_{p^{k_i}}).
\]
Thus it suffices to consider the cyclic group ${\rm C}_{p^k}=\langle g\rangle$. 
For $g^a\in {\rm C}_{p^k}$, we have
\[
(g^a)^{p^t}=e
\iff g^{ap^t}=e
\iff p^k\mid ap^t.
\]
Hence
\[
\alpha_t({\rm C}_{p^k})
=
\begin{cases}
p^t & (t<k),\\
p^k & (t\ge k)
\end{cases}
=
p^{\min(k,t)}.
\]
Therefore,
\[
\alpha_t(G_\lambda)
=
\prod_{i=1}^r p^{\min(k_i,t)}
=
p^{\sum_{i=1}^r \min(k_i,t)}
=
p^{s_t(\lambda)}.
\]
\end{proof}

\begin{lem}\label{lem:6}
For every $t=0,1,\ldots,N$, we have
\[
\varphi_{G_\lambda}(p^t)=\alpha_t(G_\lambda)-\alpha_{t-1}(G_\lambda).
\]
\end{lem}

\begin{proof}
For $t = 0$, we have
\[
\varphi_{G_\lambda}(1)=1
\]
and
\[
\alpha_0(G_\lambda)-\alpha_{-1}(G_\lambda)=1-0=1.
\]
Thus the assertion holds for $t=0$. 
Let $1\leq t\leq N$. 
Since
\[
\{g\in G_\lambda\mid \ord(g)=p^t\}
=
\{g\in G_\lambda\mid g^{p^t}=e\}\setminus
\{g\in G_\lambda\mid g^{p^{t-1}}=e\},
\]
we have
\[
\varphi_{G_\lambda}(p^t)
=
\alpha_t(G_\lambda)-\alpha_{t-1}(G_\lambda).
\]
This completes the proof.
\end{proof}

\begin{proof}[Proof of Theorem~\ref{thm:1}]
Let $n=|G_\lambda|=p^N$. 
By equation~\eqref{eq:3},
\[
\Nu(\Per(G_{\lambda}))
=
\frac{1}{2n}\sum_{d\mid n}\varphi_{G_\lambda}(d)\binom{\frac{2n}{d}}{\frac{n}{d}}
=
\frac{1}{2p^N}\sum_{t=0}^{N}\varphi_{G_\lambda}(p^t)\binom{2p^{N-t}}{p^{N-t}}.
\]
By Lemma~\ref{lem:6}, we have $\varphi_{G_\lambda}(p^t)=\alpha_t(G_\lambda)-\alpha_{t-1}(G_\lambda)$, so by Lemma~\ref{lem:5},
\begin{align*}
\Nu(\Per(G_{\lambda}))
&=
\frac{1}{2p^N}\sum_{t=0}^{N}\bigl(\alpha_t(G_\lambda)-\alpha_{t-1}(G_\lambda)\bigr)\beta_t \\
&=
\frac{1}{2p^N}
\left(
\alpha_N(G_\lambda)\beta_N
+
\sum_{t=0}^{N-1}\alpha_t(G_\lambda)(\beta_t-\beta_{t+1})
\right) \\ 
&=
1+\frac{1}{2p^{N}}\sum_{t=0}^{N-1}p^{s_t(\lambda)}(\beta_t-\beta_{t+1}).
\end{align*}
This completes the proof.
\end{proof}

\section{Order comparison via Young diagrams}
\label{sec:order}

In this section, we prove Theorem~\ref{thm:4}. 
The explicit formula obtained in Theorem~\ref{thm:1} shows that the comparison of $\Nu(\Per(G_{\lambda}))$ and $\Nu(\Per(G_{\mu}))$ is controlled by the partial column sums of the Young diagrams of $\lambda$ and $\mu$. 
We begin with two estimates that will be used to control the terms appearing in this formula.

\begin{lem}\label{lem:7}
For $t=0,1,\ldots,N-1$, we have
\[
\beta_t-\beta_{t+1}>0.
\]
\end{lem}
\begin{proof}
It suffices to show that the central binomial coefficient $\binom{2m}{m}$ is strictly increasing for $m\ge 1$. 
Indeed,
\[
\frac{\binom{2(m+1)}{m+1}}{\binom{2m}{m}}
=
\frac{(2m+2)(2m+1)}{(m+1)^2}
=
\frac{2(2m+1)}{m+1}>1.
\]
Hence $\binom{2m}{m}$ is strictly increasing in $m$. 
\end{proof}

\begin{lem}[{\cite[Lemma~5~(ii), special case]{MR4826770}}]\label{lem:8}
Let $n$ be a positive integer, and let $a,b$ be divisors of $n$ such that $b \geq 2a$. 
Then
\[
a \binom{\frac{2n}{a}}{\frac{n}{a}} > n \binom{\frac{2n}{b}}{\frac{n}{b}}.
\]
In particular, if we put $n = p^{N}$, $a = p^{m}$, and $b = p^{m+1}$, then
\[
p^{m} \beta_{m} > n \beta_{m + 1}.
\]
\end{lem}

\begin{proof}[Proof of Theorem~\ref{thm:4}]
Let $n=|G_\lambda|=|G_\mu|=p^N$. 
Then, by \eqref{eq:3},
\[
2p^N\bigl(\Nu(\Per(G_{\mu}))-\Nu(\Per(G_{\lambda}))\bigr)
=
\sum_{t=0}^{N} \gamma_t\, \beta_t, 
\]
where 
\[
\gamma_t:=\varphi_{G_\mu}(p^t)-\varphi_{G_\lambda}(p^t), 
\qquad t=0,1,\ldots,N. 
\]
We prove that 
\[
\sum_{t=0}^{N} \gamma_t\, \beta_t > 0
\]
under the assumptions 
\[
s_t(\lambda)=s_t(\mu)\quad (t=1,2,\ldots,m-1),
\qquad
s_m(\lambda)<s_m(\mu). 
\]
By the first assumption and Lemma~\ref{lem:5}, we have
\[
\alpha_t(G_\lambda)=\alpha_t(G_\mu)\qquad (t=0,1,\ldots,m-1).
\]
Therefore, by Lemma~\ref{lem:6}, we have $\gamma_t=0$ for $t<m$. 
Moreover,
\begin{align*}
\gamma_m
&=
\bigl(\alpha_m(G_\mu)-\alpha_{m-1}(G_\mu)\bigr)
-
\bigl(\alpha_m(G_\lambda)-\alpha_{m-1}(G_\lambda)\bigr) \\
&=
\alpha_m(G_\mu)-\alpha_m(G_\lambda) \\ 
&=
p^{s_m(\mu)}-p^{s_m(\lambda)}.
\end{align*}
Hence
\[
\gamma_m
=
p^{s_m(\lambda)}\bigl(p^{s_m(\mu)-s_m(\lambda)}-1\bigr)
\ge p^{s_m(\lambda)}
\ge p^m.
\]
Now define
\[
I^{-} := \{ t \in \{ m + 1, m + 2, \ldots, N \} \mid \gamma_{t} < 0 \}.
\]
Then
\[
\sum_{t=0}^{N} \gamma_t\, \beta_t
=
\gamma_m\, \beta_m
+
\sum_{\substack{t = m + 1 \\ \gamma_{t} \geq 0}}^{N} \gamma_{t} \beta_{t}
+
\sum_{t\in I^{-}} \gamma_{t} \beta_{t}.
\]
Since the second term is nonnegative, we obtain
\[
\sum_{t=0}^{N} \gamma_{t} \beta_{t}
\ge
\gamma_{m} \beta_{m} + \sum_{t \in I^{-}} \gamma_{t} \beta_{t}.
\]
Furthermore, by Lemma~\ref{lem:7},
\[
\beta_{m+1}>\beta_{m+2}>\cdots>\beta_N>0.
\]
Thus for any $t \in I^-$, it follows that
\[
\gamma_t\, \beta_t \ge \gamma_t\, \beta_{m+1}.
\]
Therefore,
\[
\sum_{t = 0}^{N} \gamma_{t} \beta_{t}
\ge
\gamma_{m} \beta_{m} + \beta_{m+1}\sum_{t \in I^{-}} \gamma_{t}
=
\gamma_{m} \beta_{m} - \beta_{m+1}\sum_{t \in I^{-}}(- \gamma_{t}).
\]
Here, 
we have
\[
-\gamma_t
=
\varphi_{G_\lambda}(p^t) - \varphi_{G_\mu}(p^t)
\le
\varphi_{G_\lambda}(p^t)
\]
for any $t = 0, 1, \ldots, N$. 
Hence
\[
\sum_{t\in I^-}(-\gamma_t)
\le
\sum_{t=0}^{N}\varphi_{G_\lambda}(p^t)
=
|G_\lambda|
=
n.
\]
Therefore, by Lemma~\ref{lem:8},
\[
\sum_{t=0}^{N} \gamma_t\, \beta_t
\ge
\gamma_{m} \beta_{m} - n \beta_{m+1}
\ge
p^m \beta_{m} - n \beta_{m+1} 
> 0.
\]
Thus
\[
2p^N \left(\Nu(\Per(G_{\mu})) - \Nu(\Per(G_{\lambda})) \right) > 0,
\]
and hence
\[
\Nu(\Per(G_{\lambda})) < \Nu(\Per(G_{\mu})).
\]
This completes the proof.
\end{proof}

\section{Congruence properties and Wolstenholme primes}
\label{sec:wolstenholme}

In this section, we study congruence properties of $\Nu(\Per(G_{\lambda}))$ for abelian $p$-groups. 
We first treat noncyclic abelian $p$-groups and then record the corresponding statement for cyclic $p$-groups. 

\begin{lem}[{\cite[Eq.~(39)]{mestrovic2011wolstenholmes}}]\label{lem:12}
Let $p \geq 5$ be a prime. Then
\[
\binom{2 p^{m}}{p^{m}} \equiv \binom{2 p^{m-1}}{p^{m-1}} \pmod{p^{3m}}.
\]
\end{lem}

Applying Lemma~\ref{lem:12} with $m=N-t$, we obtain the following corollary. 

\begin{cor}\label{cor:vp}
Let $p \geq 5$ be a prime, and let $v_p$ denote the $p$-adic valuation. 
Then
\[
v_{p}(\beta_{t} - \beta_{t+1}) \geq 3 (N-t)
\]
for every $t=0,1,\ldots,N-1$.
\end{cor}

\begin{proof}[Proof of Theorem~\ref{thm:5}]
By Theorem~\ref{thm:1},
\[
\Nu(\Per(G_{\lambda}))-1
=
\frac{1}{2p^N}\sum_{t=0}^{N-1} p^{s_t(\lambda)}(\beta_t-\beta_{t+1}).
\]
We estimate the $p$-adic valuation of each summand. 
First, consider the case $t=0$.
Since $G_\lambda$ is noncyclic, we have $N\ge 2$. 
Hence by Corollary~\ref{cor:vp}, 
\[
v_p\!\left(p^{s_0(\lambda)}(\beta_0-\beta_1)\right)
=
v_p(\beta_0-\beta_1)
\ge 3N \ge N+4.
\]
Therefore,
\[
v_p\!\left(\frac{p^{s_0(\lambda)}(\beta_0-\beta_1)}{2p^N}\right)\ge 4.
\]
Next, let $1\le t\le N-1$.
Since the partition $\lambda$ has at least two parts, 
we have 
\[
s_t(\lambda)\ge t+1.
\]
Hence by Corollary~\ref{cor:vp}, 
\[
v_p\!\left(p^{s_t(\lambda)}(\beta_t-\beta_{t+1})\right)
\ge s_t(\lambda)+3(N-t)
\ge 3N-2t+1 
= N+1+2(N-t)
\ge N+3.
\]
Therefore,
\[
v_p\!\left(\frac{p^{s_t(\lambda)}(\beta_t-\beta_{t+1})}{2p^N}\right)\ge 3.
\]
Thus each summand in
\[
\frac{1}{2p^N}\sum_{t=0}^{N-1} p^{s_t(\lambda)}(\beta_t-\beta_{t+1})
\]
is divisible by $p^3$, and hence so is the whole sum. Therefore
\[
\Nu(\Per(G_{\lambda}))\equiv 1 \pmod{p^3}.
\]
This completes the proof. 
\end{proof}

\begin{proof}[Proof of Theorem~\ref{thm:6}]
Again by Theorem~\ref{thm:1},
\[
\Nu(\Per(G_{\lambda}))-1
=
\frac{1}{2p^N}\sum_{t=0}^{N-1} p^{s_t(\lambda)}(\beta_t-\beta_{t+1}).
\]
As shown in the proof of Theorem~\ref{thm:5}, 
\[
v_{p}(p^{s_{0}(\lambda)} (\beta_{0} - \beta_{1})) \geq N + 4, 
\]
and, for $1 \leq t \leq N-2$, 
\[
v_{p}(p^{s_{t}(\lambda)} (\beta_{t} - \beta_{t+1})) 
\geq N+1+2(N-t)
\geq N+5.
\]
Since the partition $\lambda$ has at least two parts, we also have
\[
s_{N-1}(\lambda)=N.
\]
Therefore 
\[
\Nu(\Per(G_{\lambda})) - 1 \equiv \frac{1}{2} (\beta_{N-1} - \beta_{N}) \pmod{p^{4}}.
\]
Since
\[
\frac{1}{2}(\beta_{N-1} - \beta_N) 
=\frac{1}{2}\left(\binom{2p}{p}-\binom{2}{1}\right)=\binom{2p-1}{p-1}-1,
\]
we have
\[
\Nu(\Per(G_{\lambda})) \equiv \binom{2p-1}{p-1} \pmod{p^{4}}.
\]
This completes the proof. 
\end{proof}

\begin{proof}[Proof of Theorem~\ref{thm:7}]
In the cyclic case, the partition corresponding to ${\rm C}_{p^{N}}$ is $\lambda=(N)$, and hence $s_t(\lambda)=t$ for $0\leq t\leq N-1$. 
Therefore Theorem~\ref{thm:1} gives
\[
\Nu(\Per({\rm C}_{p^{N}}))-1
=
\frac{1}{2p^{N}}
\sum_{t=0}^{N-1}p^{t}(\beta_t-\beta_{t+1}). 
\]
For $0\leq t\leq N-2$, Corollary~\ref{cor:vp} gives
\[
v_p(\beta_t-\beta_{t+1})\geq 3(N-t).
\]
Hence
\[
v_p \left(
\frac{p^t(\beta_t-\beta_{t+1})}{2p^N}
\right)
\geq
t+3(N-t)-N
=
2(N-t)
\geq 4.
\]
Thus, after division by $2p^N$, all terms with $0\leq t\leq N-2$ are congruent to $0$ modulo $p^4$. 
Consequently,
\[
\Nu(\Per({\rm C}_{p^{N}}))-1
\equiv
\frac{1}{2p}(\beta_{N-1}-\beta_N)
\pmod{p^{4}}.
\]
Since
\[
\frac{1}{2}(\beta_{N-1} - \beta_N) 
=
\frac{1}{2}\left(\binom{2p}{p}-\binom{2}{1}\right)
=
\binom{2p-1}{p-1}-1,
\]
we obtain
\[
\Nu(\Per({\rm C}_{p^{N}}))-1
\equiv
\frac{1}{p}
\left\{
\binom{2p-1}{p-1}-1
\right\}
\pmod{p^{4}}.
\]
This completes the proof.
\end{proof}

\clearpage

\section*{Acknowledgments}
The first author was supported by the Strategic Priority Budget of the University of Miyazaki for the 2025 fiscal year. 

\bibliography{reference}
\bibliographystyle{plain}

\medskip
\begin{flushleft}
Naoya Yamaguchi \\
Faculty of Education, University of Miyazaki \\
1-1 Gakuen Kibanadai-nishi \\
Miyazaki 889-2192, Japan \\
{\it Email address}: n-yamaguchi@miyazaki-u.ac.jp
\end{flushleft}
\medskip

\begin{flushleft}
Yuka Yamaguchi \\
Faculty of Education, University of Miyazaki \\
1-1 Gakuen Kibanadai-nishi \\
Miyazaki 889-2192, Japan \\
{\it Email address}: y-yamaguchi@miyazaki-u.ac.jp
\end{flushleft}
\medskip

\end{document}